\documentclass[12pt]{amsart}
\usepackage[english]{babel}
\usepackage{paralist,url,verbatim, anysize}
\usepackage{amscd}
\usepackage[all,cmtip]{xy} 
\usepackage{amsmath,amsthm,amssymb,amsfonts}
\usepackage[bookmarksopen=true]{hyperref}
\usepackage[usenames, dvipsnames]{color}
\usepackage{graphicx}
\usepackage{hyperref}
\usepackage{euler, amsfonts, amssymb, latexsym, epsfig,epic}

\usepackage{tikz}
\usetikzlibrary{decorations.markings,intersections,positioning,calc}

  \tikzset{mylabel/.style  args={at #1 #2  with #3}{
    postaction={decorate,
    decoration={
      markings,
      mark= at position #1
      with  \node [#2] {#3};
 } } } }

\makeatletter
\let\@fnsymbol\@arabic
\makeatother

\theoremstyle{plain}
\newtheorem{theorem}{\bf Theorem}[section]
\newtheorem{example}[theorem]{\bf Example}
\newtheorem{conjecture}[theorem]{Conjecture}

\newtheorem{corollary}[theorem]{Corollary}
\newtheorem{lemma}[theorem]{Lemma}
\newtheorem{proposition}[theorem]{Proposition}

\newtheorem{remark}[theorem]{Remark}

\theoremstyle{definition}

\newtheorem{definition}[theorem]{Definition}

\newtheorem*{theorem*}{\bf Theorem}

\newcommand{\init}{\operatorname{in} }
\newcommand{\chara}{\operatorname{char} }

\newcommand{\N}{\mathbb{N}}

\newcommand{\PP}{\mathbb{P}}

\newcommand{\VV}{\mathbb{V}}

\newcommand{\Proj}{\operatorname{Proj} }

\newcommand{\mm}{\mathfrak{m}}

\renewcommand{\O}{\mathcal{O}}

\definecolor{mypink}{RGB}{215, 5, 234}

 \newcommand{\mnf}{\mathrm{mnf}}

\newcommand{\supp}{\operatorname{supp}}

\begin{document}

\title{Smooth Herzog projective curves}
\author{Hang Huang}
\email{hhuang235@tamu.edu}
\address{Math Department, Texas A\&M University, USA} 
\author{Yevgeniya Tarasova}
\email{ytarasov@umich.edu}
\address{Math Department, University of Michigan, USA} 
\author{Matteo Varbaro} 
\email{matteo.varbaro@unige.it}
\address{Dipartimento di Matematica, Universit\'a di Genova, Italy} 
\author{Emily E. Witt}
\email{witt@ku.edu}
\address{Department of Mathematics, University of Kansas, USA}

\begin{abstract}
    In this paper, we prove that smooth projective curves admitting a squarefree Gr\"obner degeneration have genus 0.
\end{abstract}
\date{}

\maketitle

\section{Introduction}
Let $S=K[X_0,\dots,X_n]$ be a polynomial ring over a field $K$, and $I\subseteq S$ an ideal. We say that $I$ is a {\it Herzog ideal} if there exists a monomial order $<$ such that $\init_<(I)$ is squarefree. In this paper, we are interested in this property in the case that $I$ is homogeneous (with respect to the standard grading), so that we can consider the projective variety $\VV(I)$ it defines in $\PP^n$. We say that $\VV(I)\subseteq \PP^n$ is a {\it Herzog projective variety} (with respect to the given embedding) precisely if $I\subseteq S$ is a Herzog  ideal. It is also useful to call a projective variety $X$ a {\it Herzog projective variety} if it is a Herzog projective variety with respect to some embedding.

The class of Herzog ideals/projective varieties is largely populated:  It includes Grassmannians, generic, generic symmetric and generic Hankel determinantal varieties, generic Pfaffian varieties, matrix Schubert varieties, rational normal scrolls, binomial edge ideals, and many more. 
One motivation for studying this entire class simultaneously is described by Herzog's conjecture, resolved in \cite{CoVa}, saying that the connection between $I$ and $\init_<(I)$ is much tighter than usual if $\init_<(I)$ is squarefree. Moreover, if $K$ has positive characteristic, there is a strong connection--still not completely understood--between Herzog ideals and $F$-singularities; see \cite{KoVa}.

The main result of this paper is the following theorem.

\begin{theorem}\label{maintintro}
A connected Herzog projective curve, smooth over $K$, has genus $0$.
\end{theorem}

This answers, positively, a question raised in \cite{CoVa}, and later conjectured in \cite{CDV}, in the case of curves. 
This conjecture is equivalent to the following. 

\begin{conjecture}\label{maincintro}
If $X$ is a connected Herzog projective variety, smooth over $K$, then: 
\[H^0(X,\O_X)=K, \ \ \ H^1(X,\O_X)=\cdots =H^{\dim X}(X,\O_X)=0.\]
\end{conjecture}

Another point of view, which indeed is the one on which our proof is focused, is based on the concept of Gr\"obner smoothability: an equivalent formulation of Theorem \ref{maintintro} is that the only Gr\"obner smoothable graphs are trees. If one is interested only in finding a smoothing (possibly not Gr\"obner), the story is completely different: a 4-cycle is obviously smoothable, because its Stanley-Reisner ideal is a complete intersection--and in general, many combinatorial manifolds are smoothable (see \cite{AlCh})--while according to Conjecture \ref{maincintro},- they are not Gr\"obner smoothable (see Proposition \ref{pequiv}). 

\begin{remark}
The coordinate ring of a projective Herzog variety is $F$-injective in positive characteristic and DuBois in characteristic zero by \cite[Corollary 4.11]{KoVa}, using \cite{Sch} for the characteristic zero case. However, there are plenty of smooth projective curves of positive genus with $F$-injective/DuBois coordinate rings, so Theorem \ref{maintintro} implies that the property of being Herzog is much stronger than that of being $F$-injective/DuBois.

Moreover, note that the implication ``$X$ reduced and connected $\implies$ $H^0(X,\O_X)=K$" is not true in general if $K$ is not algebraically closed. However, this implication is true if $X$ is additionally assumed to be a Herzog projective variety by \cite{Va09}.
\end{remark}

\subsection*{Acknoledgements}
H.H. is supported by NSF grant DMS-2302375; M.V. is supported by PRIN~2020355B8Y, by HubLife Science –
Digital Health (LSH-DH) PNC-E3-2022-23683267 - Progetto DHEAL-COM – CUP:D33C22001980001, and by MIUR Excellence Department Project CUP D33C23001110001; E.W. is supported by NSF CAREER DMS-1945611; all authors were supported by NSF grant DMS-1928930 and by Alfred P. Sloan Foundation grant G-2021-16778, while in residence at SLMath/MSRI, Berkeley, during the Spring 2024 Commutative Algebra program.

\section{Notations and preliminary results}

Let $S=K[X_0,\dots,X_n]$ be a polynomial ring over a field $K$, with the standard graded structure; i.e., $\deg(X_i)=1$ for all $0 \leq i \leq n$. Given a polynomial $f\in S$, we use $\supp(f)$ to denote the set of monomials occurring in $f$ with a nonzero coefficient. Let $I \subseteq S$ be a homogeneous Herzog ideal, and fix a monomial order $<$ such that $\init_<(I)$ is squarefree, so that $\init_<(I)=I_{\Delta}$ is the Stanley-Reisner ideal of a simplicial complex $\Delta$ on the vertex set $[n]_0:=\{0,\ldots ,n\}$. In particular, note that $I$ is radical. Let $X=\VV(I)\subseteq \PP^n$  be the projective variety defined by $I$. We can, and will, assume throughout that $X_0>X_1>\cdots >X_n$. 

As mentioned in the introduction, a question was raised in \cite[Question 3.17]{CoVa} that was later formulated as a conjecture in \cite[Conjecture 2]{CDV} that is equivalent to Conjecture \ref{maincintro}. We start with a result that implies that the two conjectures are indeed equivalent. Before we proceed, recall that a simplicial complex $\Gamma$ on $[n]_0$ is {\it acyclic over $K$} if $\widetilde{H}^i(\Gamma;K)=0$ for all $0 \leq i \leq \dim X$, where $\widetilde{H}^*$ denotes reduced simplicial homology. 

\begin{proposition}\label{pequiv}
With the above notation, if the Herzog projective variety $X\subseteq \PP^n$ is smooth over $K$, then the following are equivalent:
\begin{enumerate}
\item $S/I$ is a rational singularity if  $\chara(K)=0$ and $F$-rational if $\chara(K)>0$.
\item $S/I$ is Cohen-Macaulay and has negative $a$-invariant.
\item $S/I_{\Delta}$ is Cohen-Macaulay and has negative $a$-invariant.
\item $\Delta$ is acyclic over $K$.
\item $H^0(X,\O_X)=K$ and $H^i(X,\O_X)=0$ for all $1 \leq 1 \leq \dim X$.
\end{enumerate} 
\end{proposition}
\begin{proof}
Note that, since $X$ is smooth over $K$, $S/I$ is generalized Cohen-Macaulay; namely $H_{\mm}^j(S/I)_\ell=0$ for all $\ell\ll 0$ and $j<\dim S/I$, where $\mm=(X_0,\ldots ,X_n)\subseteq S$. Using \cite[Theorem 1.3]{CoVa}, $H_{\mm}^j(S/I_{\Delta})_\ell=0$ for all $\ell\ll 0$ and $j<\dim S/I=\dim S/I_{\Delta}$ as well. But then, since $S/I_{\Delta}$ is a Stanley-Reisner ring, $H_{\mm}^j(S/I_{\Delta})_\ell=0$ for all $\ell< 0$ and $j<\dim S/I_{\Delta}$, and so $H_{\mm}^j(S/I)_\ell=0$ for all $\ell< 0$ and $j<\dim S/I$. Therefore we have
\[H_{\mm}^j(S/I)_\ell=H_{\mm}^j(S/I_{\Delta})_\ell=0 \ \ \ \forall \ \ell < 0, \ j<\dim S/I=\\dim S/I_{\Delta}.\]

\medskip

$(1) \iff (2)$. This follows from Flenner's criterion \cite[Satz 3.1]{Fle} if $\chara(K)=0$ and its positive characteristic analog \cite{FeWa} if $\chara(K)>0$: in this case, the further assumption that $S/I$ is $F$-injective is true by \cite[Corollary 4.11]{KoVa} since $I\subseteq S$ is a Herzog ideal.

\medskip

$(2)\implies (5)$ follows because $H^i(X,\O_X)=H_{\mm}^{i+1}(S/I)_0$ for all $i>0$ and $\dim_KH^0(X,\O_X)=\dim_KH_{\mm}^{1}(S/I)_0+1$.

\medskip

$(5) \implies (4)$. This follows from \cite[Theorem 1.3]{CoVa} and Hochster's formula (see, e.g., \cite[Theorem 5.3.8]{BrHe}) since $H^i(X,\O_X)=H_{\mm}^{i+1}(S/I)_0$ for all $i>0$, $\dim_KH^0(X,\O_X)=\dim_KH_{\mm}^{1}(S/I)_0+1$, and $\widetilde{H}^i(\Delta;K)=H_{\mm}^{i+1}(S/I_{\Delta})_0$.

\medskip

$(4)\implies (3)$. Again using Hochster's formula $\widetilde{H}^i(\Delta;K)=H_{\mm}^{i+1}(S/I_{\Delta})_0$ and $H_{\mm}^{j}(S/I_{\Delta})_\ell=0$ for all $j, \ell>0$. Finally, we clearly have that $H_{\mm}^{0}(S/I_{\Delta})=0$ obviously, so the discussion at the beginning of the proof completes the argument.

\medskip

$(3)\implies (2)$ follows by semicontinuity.
\end{proof}

The formulation used in \cite{CoVa,CDV} is (3), and that given in the introduction as Conjecture \ref{maincintro} is (6).  To provide the proof in the curve case, we will use (5), with the aid of the following definition:  We say that $\Gamma$ is {\it Gr\"obner smoothable over $K$} if there exists a homogeneous ideal $J\subseteq S$ and a monomial order $<$ such that $\init_<(J)=I_{\Gamma}$ and $\VV(J)\subseteq \PP^n$ is connected and smooth over $K$.

\begin{conjecture}\label{conj:main}
If a $d$-dimensional simplicial complex $\Gamma$ is Gr\"obner smoothable over $K$, then $\Gamma$ is acyclic over $K$.
\end{conjecture}

To study the above conjecture, it is harmless to assume that $K$ is algebraically closed: indeed $\init_<(I\otimes_K\overline{K})=\init_<(I)\otimes_K\overline{K}$, and $X$ is smooth over $K$ if and only if $X\times_K\overline{K}$ is smooth over $\overline{K}$ (namely, it is nonsingular). As we will see, one can also assume that $\{v\}\in\Delta$ for any $v\in[n]_0$, i.e., that $I$ contains no linear forms:

\begin{lemma}
Given any $d\in\N$, if Conjecture \ref{conj:main} is true for all $d$-dimensional simplicial complexes $\Delta$ such that $\{v\}\in\Delta$ for all $v\in[n]_0$, then it is true for any $d$-dimensional simplicial complex.
\end{lemma}
\begin{proof}
Suppose that there exists $v\in[n]_0$ such that $\{v\}\notin\Delta$. Then $X_v\in I_{\Delta}=\init_<(I)$, so there exists a linear form 
\[\ell=X_v+\sum_{i>v}\lambda_iX_i\in I, \ \ \ \ \ \lambda_i\in K.\]
Consider $S/\ell S \cong K[X_j:j\in[n]_0\setminus \{v\}]=:S'$ and $I'=I/\ell S\subseteq S'$. The only thing to note is that $\init_{<'}(I')$, where $<'$ is the obvious restriction of $<$ to $S'$, is still equal to $I_{\Delta}$, where now $\Delta$ is considered on the vertex set $[n]_0\setminus \{v\}$: this can be seen because a Gr\"obner bases of $I$ remains a Gr\"obner bases  after substituting $X_v$ with $-\sum_{i>v}\lambda_iX_i$, since $i>v\iff X_v<X_i$. Since $S'/I'\cong S/I$, it is harmless to assume that every $v\in [n]_0$ belongs to $\Delta$, or equivalently, that $I$ does not contain any linear form.
\end{proof}

Our goal is to prove Conjecture \ref{conj:main} when $d=1$, and in view of the above lemma, we can assume that $\{v\}\in\Delta$ for all $v\in[n]_0$.  Hence from now on, we will assume that:

\begin{enumerate}
\item $K$ is algebraically closed.
\item $\Delta$ is a 1-dimensional simplicial complex, namely, a graph.
\item For all $v\in [n]_0$, $v$ is a vertex of $\Delta$. 
\end{enumerate}

In \cite{CDV}, some evidence for the validity of Conjecture \ref{conj:main} was given. In particular, we will use \cite[Theorem 4.9]{CDV}, encompassed in the statement below, as a starting point. Note that the point \[P_0=[1:0:\cdots :0]\in \PP^n\] is necessarily a point of $X$. Recall that a vertex $v$ of a graph $\Delta$ is a {\it free} vertex if it belongs  to exactly one edge.

\begin{theorem}[{Cf.\,\cite[Theorem 4.9]{CDV}}]\label{thm0}
With the above notation, assume that $X\subseteq \PP^n$ is a connected projective curve. Then $P_0$ is a nonsingular point of $X$ if and only if $0$ is a free vertex of the graph $\Delta$.
\end{theorem}
\begin{proof}
The``only if'' statement comes from \cite[Theorem 4.9]{CDV}. The ``if'' part is easier: let $a\in[n]$ such that $\{0,a\}$ is an edge of $\Delta$, and $g_{0j}\in I$ polynomials with $\init_<(g_{0j})=X_0X_j$ for any $j\in [n]\setminus \{a\}$. Complete the $g_{0j}$'s to a system of generators of $I$, consider the corresponding Jacobian matrix and its $(n-2)$-submatrix corresponding to the rows indexed by $g_{0j}$ and columns by $X_j$ with $j\in [n]\setminus \{a\}$. When evaluating it at $P_0$, such a submatrix can be arranged to an upper triangular matrix with $1$'s on the diagonal, so the Jacobian matrix evaluated at $P_0$ has rank $n-1$; hence $P_0$ is a nonsingular point of $X$ by the Jacobian criterion.
\end{proof}

\begin{remark}
The ``only if" part  of the above result is easy when $<$ is the lexicographic order: In this case, a monomial divisible by $X_0$ is in the support of a polynomial $f\in S$ if and only if $\init_<(f)$ is a multiple of $X_0$, and no monomial divisible by $X_0^2$ can be in the support of a polynomial $g\in I$ of the reduced Gr\"obner basis of $I$. Therefore, the Jacobian matrix, computed with respect to the reduced Gr\"obner basis of $I$, evaluated at $P_0$, has exactly $n-\delta$ nonzero rows, where $\delta$ is the number of neighbors of $0$. If $<$ is not LEX, the latter fact is not true, so one needs a more sophisticated argument in general, which is given in \cite{CDV}.
\end{remark}

As mentioned, we are going to use Theorem \ref{thm0} as the starting point for our proof of Conjecture \ref{conj:main} in dimension $1$. To settle an inductive procedure, we need to introduce the following recursive definition:

\begin{enumerate}
\item Start with $\ell:=0$ and $A:=\emptyset$.
\item Let $W=[n]_0\setminus A$, and consider $\Delta_ W$, the induced graph on $W$.
\item If there exists a free vertex $v$ of $\Delta_W$, let $\ell:=\ell+1$, $A:=A\cup\{v\}$ and go to (2). \\
Otherwise, return $\ell(\Delta):=\ell$, $A(\Delta):=A$ and $W(\Delta):=W$ as outputs.
\end{enumerate}

It is easy to show that $\ell(\Delta)$ is an invariant of the graph, namely that it does not depend on the choice of the free vertices $v$ in the recursion. The number $\ell(\Delta)$ is indeed equal to the maximum number of collapses that it is possible perform, starting from $\Delta$. As it turns out, $\ell(\Delta)\leq n$, and equality holds if and only if $\Delta$ is a tree. If $\Delta$ is connected and it is not a tree, $A(\Delta)$ and $W(\Delta)$ are invariants of $\Delta$ as well.

Theorem \ref{thm0} implies that $X$ is singular when $\Delta$ is a graph with $\ell(\Delta)=0$.  We will carry out an inductive argument on $\ell(\Delta)$ to prove that $X$ is singular whenever $\ell(\Delta)<n$; this will prove Conjecture \ref{conj:main} for curves.

\begin{remark}
    Given a graph $\Delta$, using the ``VersalDeformations'' package for Macaulay2 \cite{GS}, we are able to write down all possible Gr\"{o}bner deformations for $\Delta$. In the case that $\Delta$ is not Gr\"{o}bner smoothable, by analyzing the singularities of a generic Gr\"{o}bner deformation, we are able to pin down a special point in $X$ that is always singular for all Gr\"{o}bner deformations. In the following section, we will show that if $\Delta$ is not a tree ($\ell(\Delta) < n$), then up to a change of variable preserving the initial ideal, we can always show that $P_a$ is a singular point of X, where $a$ is the largest variable in $W(\Delta)$.  
\end{remark}

\section{Proof of the main result}

We denote by $\mnf(\Delta)$ the set of minimal non-faces of $\Delta$, namely,
\[\mnf(\Delta)=\{\sigma\notin\Delta:\tau\in\Delta \ \forall \ \tau\subsetneq\sigma\}.\]

Note that, since $\{v\}\in \Delta$ for all $v\in[n]_0$ and $\Delta$ is a graph, if $\sigma\in\mnf(\Delta)$, then $|\sigma|\in\{2,3\}$.
For all $\sigma\subseteq [n]_0$, let $X_{\sigma}:=\prod_{i\in\sigma}X_i$. We have that $I_{\Delta}=(X_{\sigma}:\sigma\in\mnf(\Delta))$. For each minimal non-face $\sigma$ of $\Delta$, let $g_{\sigma}\in I$ be a homogeneous (quadratic or cubic, depending on $|\sigma|$) polynomial such that $\init_<(g_{\sigma})=X_{\sigma}$. The set
\[G=\{g_{\sigma}\in S:\sigma\in\mnf(\Delta)\}\]
is a Gr\"obner basis of $I$. Sometimes, if $\sigma=\{i,j\}$, we will write $g_{ij}$ for $g_{\sigma}$. We will not assume that $G$ is reduced, since for some arguments, we will perform some change of variables after which the property of being reduced may be lost. However, we will need the following weaker, but easier to control, notion:

\begin{definition}
We say that a Gr\"obner basis $f_1,\ldots ,f_m\in S$ is {\it $0$-reduced} if $\init_<(f_i)\notin \supp(f_j)$ for all $i\neq j$ such that $X_0|\init_<(f_i)$.
\end{definition}

\begin{lemma}\label{l1}
Assume that $0$ is a free vertex of $\Delta$ and that $G$ is $0$-reduced. Then, for all $g\in G$, $\init_<(g)\in K[X_1,\ldots ,X_n]\implies g\in K[X_1,\ldots ,X_n]$. Hence the set $G'=\{g_{\sigma}\in S:\sigma\in\mnf(\Delta),0\notin \sigma\}$ is Gr\"obner basis of the ideal $I'=I\cap K[X_1,\ldots ,X_n]$.
\end{lemma}
\begin{proof}
Let $b\in [n]$ be the only vertex such that $\{0,b\}\in\Delta$. Since $G$ is $0$-reduced, it is enough to show that 
\[P(\sigma): \ \ \ X_0^2,X_0X_b,X_0^3,X_0^2X_b,X_0X_b^2\notin\supp(g_{\sigma})\] 
when $\sigma\subseteq [n]$ is a minimal non-face of $\Delta$. Note that the fact that $X_0^2,X_0^3\notin\supp(g_{\sigma})$ for each $\sigma\in\mnf(\Delta)$ is trivial. We introduce, for $\sigma\in\mnf(\Delta)$, the following property:
\[P'(\sigma): \ \ \ X_0X_b,X_0^2X_b\notin\supp(g_{\sigma}).\]
If $\alpha$ and $\beta$ are minimal non-faces of $\Delta$, we write $\alpha\prec\beta$ if:
\[|\alpha|=|\beta| \ \mbox{ and } \ X_{\alpha}<X_{\beta} \ \ \ \mbox{ or } \ \ \ |\alpha|<|\beta| \ \mbox{ and } \ X_0X_{\alpha}<X_{\beta} \ \ \ \mbox{ or } \ \ \ |\alpha|>|\beta| \ \mbox{ and } \ X_{\alpha}<X_0X_{\beta}.\]

\medskip

Let $\tau \in \mnf(\Delta)$ such that $X_0^2 X_b \in \supp(g_{\tau})$, but $P'(\sigma)$ holds whenever when $\sigma\subseteq [n]$ is a minimal non-face of $\Delta$ such that $\sigma\prec\tau$. Then $\tau=\{i,j,k\}$ with $0<i<j<k<b$. In particular $\{0,k\}\in \mnf(\Delta)$, so we can consider the $S$-polynomial between $g_{0k}$ and $g_{\tau}$:
\[f=X_iX_jg_{0k}-X_0g_{\tau}\in I.\]
Note that $X_0^3X_b\in\supp(f)$. On the other hand, one should be able to reduce $f$ modulo $G$:
\[f=\sum_{\sigma\in \mnf(\Delta)}q_{\sigma}g_{\sigma} \ \ \ \mbox{ with }\init_<(q_{\sigma}g_{\sigma})< X_0X_iX_jX_k,\]
where $q_{\sigma}$ are quadrics if $|\sigma|=2$ and linear forms if $|\sigma|=3$. In particular, there should exist $\sigma\in \mnf(\Delta)$ such that $X_0^3X_b$ is in the support of $q_{\sigma}g_{\sigma}$ and  $\init_<(q_{\sigma}g_{\sigma})< X_0X_iX_jX_k$. 
\begin{enumerate}
\item If $|\sigma|=3$, since $X_0^3 \notin \supp(g_{\sigma})$, the only possibility is that $\init_<(q_\sigma)=X_0$ and $X_0^2 X_b \in \supp(g_{\sigma})$. But the latter condition implies that $\sigma\succeq \tau$, and so 
\[\init_<(q_{\sigma}g_{\sigma})=X_0\init_<(g_\sigma)\geq X_0\init_<(g_\tau)=X_0X_iX_jX_k,\]
a contradiction
\item If $|\sigma|=2$, since $X_0^2 \notin \supp(g_{\sigma})$, $\init_<(q_\sigma)=X_0^2$ and $X_0 X_b \in \supp(g_{\sigma})$. 
\begin{enumerate} 
\item If $0\notin \sigma$, $X_0 X_b \in \supp(g_{\sigma})\implies \sigma\succeq \tau$, and so 
\[\init_<(q_{\sigma}g_{\sigma})=X_0^2\init_<(g_\sigma)\geq X_0\init_<(g_\tau)=X_0X_iX_jX_k,\]
a contradiction 
\item If $0\in\sigma$, say $\sigma=\{0,q\}$, then reducing $f$ at a certain point, a term of the type $X_0^3X_q$, with $q\neq b$, will occur. But this is impossible, because 
\[f=\lambda X_0^3X_b+X_0^2g+X_0h+r\] 
where $\lambda\in K$ and $ g,h,r\in K[X_1,\ldots ,X_n]$, and reducing it, the only possibly new occurring term divisible by $X_0^3$ is $X_0^3X_b$ since $G$ is $0$-reduced. So at each step the reduction of $f$ will have the form $f'=\lambda' X_0^3X_b+X_0^2g'+X_0h'+r'$ where $\lambda'\in K$ and $g',h',r'\in K[X_1,\ldots ,X_n]$.
\end{enumerate}
\end{enumerate}

\medskip

Next, let $\tau \in \mnf(\Delta)$ such that $X_0 X_b \in \supp(g_{\tau})$, but $P'(\sigma)$ holds whenever $\sigma\subseteq [n]$ is a minimal non-face of $\Delta$ such that $\sigma\prec \tau$. Then $\tau=\{i,j\}$ with $0<i<j<b$. In particular, $\{0,j\}\in\mnf(\Delta)$, so we can consider the $S$-polynomial $f$ between $g_{0j}$ and $g_{\tau}$, namely
\[f=X_ig_{0j}-X_0g_{\tau}\in I.\]
Note that $X_0^2X_b\in\supp(f)$. On the other hand, one should be able to reduce $f$ modulo $G$:
\[f=\sum_{\sigma\in \mnf(\Delta)}\ell_{\sigma}g_{\sigma} \ \ \ \mbox{ with }\init_<(\ell_{\sigma}g_{\sigma})< X_0X_iX_j,\]
where $\ell_{\sigma}$ are linear forms if $|\sigma|=2$ and scalars if $|\sigma|=3$. In particular, there should exist a minimal non-face $\sigma$ of $\Delta$ such that $X_0^2X_b$ is in the support of $\ell_{\sigma}g_{\sigma}$ and   $\init_<(\ell_{\sigma}g_{\sigma})< X_0X_iX_j$. 
\begin{enumerate}
\item If $|\sigma|=3$, then $X_0^2 X_b \in \supp(\ell_{\sigma}g_{\sigma})=\supp(g_{\sigma})$. Then $\sigma\succeq \tau$, and so 
\[\init_<(\ell_{\sigma}g_{\sigma})=\init_<(g_\sigma)\geq X_0\init_<(g_{\tau})=X_0X_iX_j,\]
a contradiction
\item If $|\sigma|=2$, since $X_0^2 \notin \supp(g_{\sigma})$, $\init_<(\ell_\sigma)=X_0$ and $X_0 X_b \in \supp(g_{\sigma})$. 
\begin{enumerate} 
\item If $0\notin \sigma$, $X_0 X_b \in \supp(g_{\sigma})\implies \sigma\succeq \tau$, and so 
\[\init_<(\ell_{\sigma}g_{\sigma})=X_0\init_<(g_\sigma)\geq X_0\init_<(g_\tau)=X_0X_iX_j,\]
a contradiction 
\item If $0\in\sigma$, say $\sigma=\{0,q\}$, then reducing $f$ at a certain point a term of the type $X_0^2X_q$, with $q\neq b$, will occur. But this is impossible, because 
\[f=\lambda X_0^2X_b+X_0g+h\] 
where $\lambda\in K$ and $ g,h\in K[X_1,\ldots ,X_n]$, and reducing it, the only possibly new occurring term divisible by $X_0^2$ is $X_0^2X_b$ since $G$ is $0$-reduced. So at each step, the reduction of $f$ will be  of the form $f'=\lambda' X_0^2X_b+X_0g'+h'$, where $\lambda'\in K$ and $g',h'\in K[X_1,\ldots ,X_n]$.
\end{enumerate}
\end{enumerate}

\medskip

Hence $P'(\sigma)$ holds for any minimal non-face $\sigma\subseteq [n]$ of $\Delta$.

\medskip

Finally, let $\tau\in \mnf(\Delta)$ such that $X_0 X_b^2 \in \supp(g_{\tau})$, but $X_0 X_b^2 \notin \supp(g_{\sigma})$ whenever $\sigma\subseteq [n]$ is in $\mnf(\Delta)$ and $X_\sigma <X_\tau$. Then $\tau=\{i,j,k\}$ with $0<i<j<b,k$. In particular, $\{0,j\}\in\mnf(\Delta)$, so we can consider the $S$-polynomial between $g_{0j}$ and $g_{\tau}$:
\[f=X_iX_kg_{0j}-X_0g_{\tau}\in I.\]
Note that $X_0^2X_b^2\in\supp(f)$. On the other hand, one should be able to reduce $f$ modulo $G$:
\[f=\sum_{\sigma\in\mnf(\Delta)}q_{\sigma}g_{\sigma} \ \ \ \mbox{ with }\init_<(q_{\sigma}g_{\sigma})< X_0X_iX_jX_k.\]
where $q_{\sigma}$ are quadrics if $|\sigma|=2$ and linear forms if $|\sigma|=3$. In particular there should exist $\sigma\in \mnf(\Delta)$ such that $X_0^2X_b^2$ is in the support of $q_{\sigma}g_{\sigma}$ and   $\init_<(q_{\sigma}g_{\sigma})< X_0X_iX_jX_k$. 
\begin{enumerate}
\item If $|\sigma|=3$, since $X_0^3,X_0^2X_b \notin \supp(g_{\sigma})$, the only possibility is that $\init_<(q_\sigma)=X_0$ and $X_0 X_b^2 \in \supp(g_{\sigma})$. But the latter condition implies that $X_\sigma\geq X_ \tau$, and so 
\[\init_<(q_{\sigma}g_{\sigma})=X_0\init_<(g_\sigma)\geq X_0\init_<(g_\tau)=X_0X_iX_jX_k,\]
a contradiction.
\item If $|\sigma|=2$ and $0\notin \sigma$, since $X_0^2,X_0X_b \notin \supp(g_{\sigma})$, $\init_<(q_\sigma)=X_0^2$ and $X_b^2 \in \supp(g_{\sigma})$. Then reducing $f$ at a certain point, a term of the type $X_0^2X_\sigma$ will occur. But this is impossible, because 
\[f=\lambda X_0^2X_b^2+X_0g+h,\] 
where $\lambda\in K$ and $ g,h\in K[X_1,\ldots ,X_n]$, and reducing it, the only possibly new occurring term divisible by $X_0^2$ is $X_0^2X_b^2$ since $G$ is $0$-reduced and $P'(\sigma)$ holds for any minimal non-face $\sigma\subseteq [n]$ of $\Delta$. So at each step, the reduction of $f$ will be  of the form $f'=\lambda' X_0^2X_b^2+X_0g'+h'$, where $\lambda'\in K$ and $g',h'\in K[X_1,\ldots ,X_n]$.
\item If $|\sigma|=2$ and $0\in\sigma$, say $\sigma=\{0,q\}$; since $X_0^2\notin \supp(g_{\sigma})$ there are two possibilities: $\init_<(q_\sigma)=X_0^2$ and $X_b^2 \in \supp(g_{\sigma})$ or $\init_<(q_\sigma)=X_0X_b$ and $X_0X_b \in \supp(g_{\sigma})$. Then reducing $f$ at a certain point, a term of the type $X_0^3X_q$ (in the first case) or $X_0^2X_bX_q$ (in the second case) will occur, and both situations are impossible, arguing as in the previous point.
\end{enumerate}

\end{proof}
Assume that $0$ is a free vertex of $\Delta$ and that $G$ is $0$-reduced. The above lemma implies that if $S'=K[X_1,\ldots ,X_n]$ and $I'=I\cap S'\subseteq S'$, then 
\[G'=\{g_{\sigma}\in S:\sigma\in\mnf(\Delta) , \ \sigma\subseteq [n]\}\]
is a Gr\"obner basis of $I'$. Geometrically, $I'$ defines the closure $X'$ of the projection of $X$ from the point $P_0=[1:0:\ldots :0]\in\PP^n$ to the hyperplane $H_0=\{X_0=0\}\cong \PP^{n-1}$. It is necessary to take the closure because $P_0$ is (clearly) a point of $X$ and the projection $\pi:\PP^n\setminus \{P_0\}\to H_0$ is not defined at $P_0$.

For every integer $0\leq a \leq n$, let $P_a$ denote the point of $\PP^n=\Proj S$ with the only nonzero entry in the $a$-th position; similarly, for all $1\leq b \leq n$, let  $P'_b$ denote the point of $\PP^{n-1}=\Proj S'$ with the only nonzero entry in the $b$-th position.

\begin{lemma}\label{l2}
Assume that $0$ is a free vertex of $\Delta$. If, for some $a\in [n]$ such that $P_a'\in X'$ there exists $j\in [n]$ with $X_0 X_a \in \supp(g_{0j})$ (if $\{0,a\}\notin\Delta$ the latter condition is automatically satisfied by taking $j=a$), then there exists a unique $Q\in X$ such that $\pi(Q)=P_a'$. Furthermore, $Q$ is singular if and only if $P_a'$ is singular.
\end{lemma}
\begin{proof}
First of all, we need to show that $P_a'\in\pi(X)\subseteq X'$. The only way in which $P_a'\in X'\setminus \pi(X)$ can happen is if $P_a$ belongs to the tangent space of $X$ at $P_0$. Such a tangent space is defined by the linear forms:
\[\ell_{\sigma}=\sum_{k=1}^n\frac{\partial g_{\sigma}}{\partial X_k}(P_0)X_k, \ \ \ \sigma\in \mnf(\Delta).\]
If $j\in [n]$ is such that $X_0X_a$ occurs with nonzero coefficient $\lambda\in K$ in $g_{0j}$, then
\[\ell_{0j}(P_a)=\lambda\neq 0.\]
Hence $P_a'$ is actually in $\pi(X)$. Moreover, if $P,Q\in X$ are such that $\pi(P)=\pi(Q)=P'_a$, from $g_{0j}(P)=g_{0j}(Q)=0$ we deduce that  $P=Q=[\mu:0:\ldots :0:1:0:\ldots:0]\in\PP^n$ with $1$ at the $a$-th entry and $\mu=-\alpha/\lambda$, where $\alpha$ is the coefficient of $X_a^2$ in $g_{0j}$. Hence there exists a unique $Q\in X$ such that $\pi(Q)=P'_a$. 

Of course, if $Q$ is singular, then $\pi(Q)=P_a'$ is also singular. On the other hand, suppose that $Q$ is nonsingular. To show that $\pi(Q)=P_a'$ is nonsingular, it is enough to check that $P_0$ is not in the tangent space of $X$ at $Q$. But one of the equations cutting out this tangent space is 
\[q_{0j}=\frac{\partial g_{0j}}{\partial X_0}(Q)(X_0-\mu X_a)+\sum_{\substack{k=1 \\ q\neq a}}^n\frac{\partial g_{0j}}{\partial X_k}(Q)X_k,\]
Since $q_{0j}(P_0)=\displaystyle \frac{\partial g_{0j}}{\partial X_0}(Q)=\lambda \neq 0$, $P_0$ does not belong to the tangent space of $X$ at $Q$.
\end{proof}

\begin{theorem}\label{thm:main}
If $X$ is a nonsingular connected curve, then $\Delta$ is a tree.
\end{theorem}

\begin{proof}

The fact that $\Delta$ is a connected graph follows by \cite{KaSt95}. To prove that $\Delta$ is a tree, we will prove that $\ell(\Delta)=n$. Theorem \ref{thm0} implies that $\ell(\Delta)>0$, and we will prove by induction on $\ell(\Delta)$ that $X$ is singular if $\ell(\Delta)<n$.

\medskip

Assume by contradiction that $\Delta$ is not a tree,  and let $a\in [n]$ be the smallest vertex belonging to $W(\Delta)$, so that $X_a>X_i$ whenever $i\neq a$ belongs to $W(\Delta)$. We want to show that either $X$ is reducible or there is a change of variables preserving $\init_<(I)$ so that $P_a$ is a singular point of $X$.

\medskip

By Theorem \ref{thm0}, $0$ belongs to a unique edge of $\Delta$. So, assuming that $G$ is $0$-reduced (if not, $0$-reduce it), we are in the situation of Lemma \ref{l1}. Let $S'=K[X_1,\ldots ,X_n]$, $I'=I\cap S'\subseteq S'$, $G'=\{g_{\sigma}\in S':\sigma\subseteq [n],\sigma\in \mnf(\Delta)\}$ a Gr\"obner basis of $I'$, $X'=\VV(I')\subseteq \PP^{n-1}$ the closure of the projection of $X\setminus \{P_0\}$ from $P_0$, $<'$ the restriction of $<$ to $S'$, and $\Delta'=\Delta_{[n]}$. Note that $\init_{<'}(I')=I_{\Delta'}$, $\ell(\Delta')=\ell(\Delta)-1$ and $W(\Delta')=W(\Delta)$. If $X'$ is reducible, then $X$ must be reducible as well. Hence we can assume that $X'$ is irreducible, so we know by induction that (there exists a change of variables in $S'$ preserving $\init_{<'}(I')$ under which) $P_a'$ is a singular point of $X'$.

\bigskip

{\bf Case I: There exists $j\in [n]$ with $X_0 X_a \in \supp(g_{0j})$}: In this situation, by Lemma \ref{l2}, there exists a point $Q\in X$ of the form $Q=[\mu:0:\ldots :0:1:0:\ldots:0]$ at which $X$ is singular, contradicting the fact that $X$ is smooth. To preserve the inductive hypothesis, we make a change of variables of the form $X_0\mapsto X_0-\mu X_a$ that will preserve $\init_<(I)$ and make $P_a$ a singular point of $X$.

\bigskip

{\bf Case II: $X_0 X_a \notin \supp(g_{0j})$ for all $j\in [n]\setminus \{a\}$}: We want to show in this case that $X$ is reducible, contradicting the fact that it is nonsingular and connected. Toward this goal, we will prove that, up to a change of variables preserving the initial ideal, no monomials of $K[X_0,X_a]$ belong to $\supp(g_{\sigma})$ for any minimal non-face $\sigma$ of $\Delta$. Hence the whole line $\overline{P_0P_a}$ would be contained in the curve $X$, so that it would not be irreducible. It turns out that it is enough to prove that $X_a^2\notin\supp(g_{0j})$ for all $j\in [n]\setminus \{a\}$.

\medskip

Let $0=a_0 < a_1 <a_2 < \ldots < a_k$ be the neighbors of $a$, and note that $k\geq 2$.

\medskip

{\it Step 1}: For any $c\in [n]$ such that $\{a , c\} \notin \Delta$, perform a change of variables of the form 
\[X_c\mapsto X_c - \sum_{r=1}^k \lambda_{cr}X_{a_r},\]
where $X_aX_{a_r}$ occurs in $g_{ac}$ with coefficient $\lambda_{cr}\in K$. Note that if $\lambda_{cr}\neq 0$, we must have $X_c>X_{a_r}$, so this change of variables preserves $\init_<(I)$. 
Let us $0$-reduce the transformed Gr\"obner basis. As one can check, we have:

\begin{enumerate}
\item $X_0 X_a \notin \supp(g_{\sigma})$ for any minimal non-face $\sigma$ of $\Delta$.
\item For any $c\in [n]$ such that $\{a , c\} \notin \Delta$, $X_aX_{a_r}\notin\supp(g_{ac})$ for all $r\in[k]$.
\item $X_a^2\notin \supp(g_{\sigma})$ for any minimal non-face $\sigma\subseteq [n]$ of $\Delta$.
\end{enumerate}

We denote by $A_0$ the set $\{\{a , c\} \notin \Delta:c\in[n]\}$.

\medskip

    {\it Step 2}: If, by contradiction, there exists $j\in [n]$ such that $X_a^2 \in \supp (g_{0j})$, define a sequence $b_1,\ldots, b_{k-1}$ as follows:
    \[b_r=\begin{cases}
    a_r & \mbox{ if }j\notin\{a_1,\ldots ,a_{k-1}\} \mbox{ or }a_r<j \\
    a_{r+1} & \mbox{ if }j\in\{a_1,\ldots ,a_{k-1}\} \mbox{ and }a_r\geq j
     \end{cases}\]
In the support of the S-polynomial $S(g_{0j},g_{0b_1})\in K[X_1,\ldots ,X_n]$, there will be $X_a^2 X_{b_1}$. Since the polynomials $g_{\sigma}$, where $\sigma\in \mnf(\Delta)$ form a Gr\"obner basis, $S(g_{0j},g_{0b_1})$ must reduce to zero modulo them. By the properties listed in Step 1, 
 the only way for $X_a^2 X_{b_1}$ to be canceled is if there exists a minimal non-face $\sigma_1\subseteq [n]$ of $\Delta$ which is not in $A_0$ and  so that $X_a X_{b_1}$ (or $X_a^2X_{b_1}$, depending on whether $|\sigma_1|=2$ or $|\sigma_1|=3$) belongs to $\supp(g_{\sigma-1})$. We perform a change of variables of the form 
 \[X_{b_1}\mapsto X_{b_1}-\sum_{r=2}^{k-1}\frac{\mu_{1r}}{\mu_{11}}X_{b_r}\]
 where $X_aX_{b_r}$ (or $X_a^2X_{b_r}$) occurs in $g_{\sigma_1}$ with coefficient $\mu_{1r}\in K$. As before, this change of variables preserves $\init_<(I)$. We $0$-reduce the new Gr\"obner basis, note that all the properties listed at the end of Step 1 are preserved and that, calling $A_1=A_0\cup \{\sigma_1\}$:
 
 \[(*) \ \ \ \mbox{For } r\in \{2,\ldots ,k-1\}, \ X_aX_{b_r},X_a^2X_{b_r}\notin\supp(g_{\sigma}) \ \forall \ \sigma\in A_1.\]  
 \smallskip
 
Next, in the support of the S-polynomial $S(g_{0j},g_{0b_2})\in K[X_1,\ldots ,X_n]$, there will be $X_a^2 X_{b_2}$. Reducing it to zero modulo $G$, as before the only way for $X_a^2 X_{b_2}$ to be canceled is if there exists a minimal non-face $\sigma_2\subseteq [n]$ of $\Delta$ that is not in $A_1$, and  so that $X_a X_{b_2}$ (or $X_a^2X_{b_2}$, depending on $|\sigma_2|$) belongs to $\supp(g_{\sigma_2})$. We do a change of variables of the form 
 \[X_{b_2}\mapsto X_{b_2}-\sum_{r=3}^{k-1}\frac{\mu_{2r}}{\mu_{22}}X_{b_r}\]
where $X_aX_{b_r}$ (or $X_a^2X_{b_r}$) occurs in $g_{\sigma_2}$ with coefficient $\mu_{2r}\in K$. As before, this change of variables preserves $\init_<(I)$. We $0$-reduce the new Gr\"obner basis, note that all the properties listed at the end of Step 1 are preserved and that, calling $A_2=A_1\cup \{\sigma_2\}$:
 
 \[(**) \ \ \ \mbox{For } r\in \{3,\ldots ,k-1\}, \ X_aX_{b_r},X_a^2X_{b_r}\notin\supp(g_{\sigma}) \ \forall \ \sigma\in A_2.\] 
 
\smallskip

We can continue this way to get a subset $A_{k-1}=A_0\cup\{\sigma_1,\sigma_2, \ldots ,\sigma_{k-1}\}\subseteq \mnf(\Delta')$ of cardinality $n-2$ and a Gr\"obner basis $\{g_{\sigma}:\sigma\in\mnf(\Delta)\}$ with the following properties:.

\begin{enumerate}
\item For all $\{a , c\} \notin \Delta$, the only monomial divisible by $X_a$ in the support of $g_{ac}$ is $X_aX_c$.
\item For all $r=1,\ldots ,k-1$, the only monomials of degree 2 divisible by $X_a$ or of degree 3 divisible by $X_a^2$ potentially in the support of $g_{\sigma_r}$ are $X_aX_{b_j}$ or $X_a^2X_{b_j}$ (depending on whether $|\sigma_r|=2$ or $|\sigma_r|=3$) with $j\leq r$, and $X_aX_{b_r}$ or $X_a^2X_{b_r}$ belong for sure in the support of $g_{\sigma_r}$.
\item $X_a^2,X_a^3\notin \supp(g_{\sigma})$ for any minimal non-face $\sigma\subseteq [n]$ of $\Delta$.
\end{enumerate}

\medskip

{\it Step 3}: Consider the Jacobian matrix of $I'=(g_{\sigma}:\sigma\in\mnf(\Delta')\})=I\cap K[X_1,\ldots ,X_n]$ with rows indexed by the minimal non-faces $\sigma\subseteq [n]$ of $\Delta$ and with columns $1,\ldots ,n$. Choose the minor corresponding to the rows $g_{\sigma}$ with $\sigma\in A_{k-1}$ and the columns corresponding to the variables $X_c$ and $X_{b_r}$ for $r=1,\ldots ,k-1$. Such a minor can be arranged into a lower triangular matrix of size $(n-2) \times (n-2)$ with nonzero entries on the diagonal. This contradicts the fact that, by the inductive hypothesis on $\ell(\Delta)$, since $\ell(\Delta')=\ell(\Delta)-1<n$, $P'_a$ is a singular point in the projection $X'$. 
\end{proof}

\begin{corollary}\label{c:genus}
A connected projective Herzog curve, smooth over a field, has genus 0. Equivalently, a graph that is Gr\"obner smoothable over some field must be a tree. 
\end{corollary}

%

\begin{example}
Let $I\subseteq S$ the ideal of 2-minors of the matrix
\[\begin{pmatrix}
X_0 & X_1 & X_2 & \cdots & X_{n-1} \\
X_1 & X_2 & X_3 & \cdots & X_n
\end{pmatrix}.\]
Then $X=\VV(I)\subseteq\PP^n$ is a rational normal curve of degree $n$. If $<$ is the lexicographic order with $X_0>X_1>\cdots > X_n$, then $\init_<(I)=(X_iX_{j+1}:0\leq i<j\leq n-1)$, so $\init_<(I)=I_{\Delta}$, where $\Delta$ is a path of length $n$. It is not difficult to check that all the different initial ideals of $I$ are not squarefree, see \cite[Theorem 4.9]{CDR} for a description of all possible Cohen-Macaulay initial ideals of $I$. Since the rational normal curve is the only normally embedded projective curve smooth over an algebraically closed field of genus 0, exploiting Corollary \ref{c:genus}, one might think that the only Gr\"obner smoothable graphs are the paths. 

Unfortunately, this is not true, because initial ideals depend on changes of coordinates. As an example, consider the ideal $J\subseteq K[X_0,X_1,X_2,X_3]$  generated by the $2 \times 2$-minors of the matrix
\[\begin{pmatrix}
X_1+X_2+X_3 & X_1+X_3 & X_1 \\
X_1+X_3 & X_1 & X_0+X_1
\end{pmatrix}.\]
As one can check, for any monomial order with $X_0>X_1>X_2>X_3$ we have $\init_<(J)=(X_0X_1,X_0X_2,X_1X_2)$ (over any field $K$), i.e. $\init_<(J)=I_{\Delta}$ where $\Delta$ the tree with edges $03,13,23$--namely, a star.

We do not know which trees are Gr\"obner smoothable. Also, whether a Gr\"obner smoothing exists may depend, in principle, on the field $K$.
\end{example}

\bibliographystyle{acm}
\bibliography{AJME}
\end{document}